\documentclass[11pt,twoside]{article}
\usepackage{amsmath,amssymb,epsfig,color,url}
\usepackage{graphicx,epstopdf,subfigure}
\numberwithin{equation}{section}
\newtheorem{proposition}{Proposition}[section]
\newtheorem{theorem}[proposition]{Theorem}
\newtheorem{lemma}[proposition]{Lemma}
\newtheorem{corollary}[proposition]{Corollary}

\newtheorem{remark}[proposition]{Remark}
\newtheorem{example}[proposition]{Example}
\newtheorem{algorithm}[proposition]{Algorithm}

\newenvironment{proof}{{\noindent \em Proof.}}{\hfill $\fbox{}$ \vspace*{5mm}}

\newcommand{\ba}{{\bf a}}

\newcommand{\bfe}{{\bf e}}

\newcommand{\bu}{{\bf u}}

\newcommand{\bv}{{\bf v}}

\newcommand{\diag}{{\rm diag}}

\newcommand{\off}{{\rm off}}

\newcommand{\rank}{{\rm rank}}

\newcommand{\tol}{{\tt tol}}

\newcommand{\R}{{\mathbb R}}

\newcommand{\Rmn}{{\mathbb R}^{m\times n}}

\newcommand{\BE}{\begin{equation}}
\newcommand{\EE}{\end{equation}}

\definecolor{DarkGreen}{rgb}{0,0.6,0}

\definecolor{red}{rgb}{1,0,0}           
\definecolor{green}{rgb}{0,1,0}
\definecolor{blue}{rgb}{0,0,1}

\definecolor{light}{gray}{.98}          
\definecolor{dark}{gray}{.20}
\definecolor{pink}{rgb}{.95,0.82,0.92}  
\definecolor{backgd}{rgb}{0.65,0.73,0.7}  

\begin{document}

\title{\bf A Structure-Preserving One-Sided  Jacobi Method for Computing the SVD  of a Quaternion Matrix}
\author{Ru-Ru Ma\thanks{School of Mathematical Sciences, Xiamen University, Xiamen 361005, People's Republic of China  (maruru7271@126.com). The research of this author is partially supported by the Fundamental Research Funds for the Central Universities (No. 20720180008).}
\and Zheng-Jian Bai\thanks{Corresponding author. School of Mathematical Sciences and Fujian Provincial Key Laboratory on Mathematical Modeling \& High Performance Scientific Computing,  Xiamen University, Xiamen 361005, People's Republic of China (zjbai@xmu.edu.cn). The research of this author is partially supported by the National Natural Science Foundation of China (No. 11671337), the Natural Science Foundation of Fujian Province of China (No. 2016J01035), and the Fundamental Research Funds for the Central Universities (No. 20720180008).}}
\date{}
\maketitle

\begin{abstract}
In this paper, we  provide a structure-preserving one-sided cyclic Jacobi method for computing the singular value decomposition of a  quaternion matrix.   In this method,   the columns of the quaternion matrix  are orthogonalized in pairs by using a sequence of orthogonal JRS-symplectic Jacobi matrices to its real counterpart.  The quadratic convergence  is also established under some mild conditions.  Numerical tests are reported to illustrate the efficiency of the proposed method.
	
\end{abstract}

\vspace{3mm}
{\bf Keywords.} Quaternion matrix; singular value decomposition;  one-sided cyclic Jacobi method; Color image compression

\vspace{3mm}
{\bf AMS subject classifications.}  65F18,  65F15, 15A18, 65K05, 90C26, 90C48

\section{Introduction}\label{sec1}
The concept of quaternions was originally introduced by  Hamilton \cite{hamilton66}.
Quaternions and quaternion matrices arise in various applications in applied science such as quaternionic quantum mechanics \cite{Ar11, DM92, fj59}, color image processing \cite{ XYXN15} and field theory \cite{lr92}, etc. The quaternion matrix singular value decomposition (QSVD) was studied theoretically in 1997 by Zhang \cite{zhangf97}. Recently, the QSVD has been an important tool  in many applications such as color image processing \cite{GY15, bs03, lm18, pcd03}, signal processing \cite{bm04,YL17}, and electroencephalography \cite{EK16}, etc.

Various numerical methods have been proposed for computing the QSVD. In \cite{l01, bm04, pcd03}, some algorithms were provided to calculate the singular value decomposition (SVD) of a  quaternion matrix via utilizing its equivalent complex matrix. In \cite{sabi06}, Sangwine and Le Bihan proposed a method for computing the QSVD based on bidiagonalization via quaternionic Householder transformations. In \cite{bisa07}, Le Bihan and Sangwine gave an implicit Jacobi algorithm for computing the QSVD where the quaternion arithmetic is employed instead of a complex equivalent representation. In \cite{DO11}, Doukhnitch and Ozen presented a coordinate rotation digital computer algorithm  for computing the QSVD.

In this paper, we  propose a  structure-preserving one-sided cyclic Jacobi method for computing the QSVD. This is motivated by the recent  structure-preserving methods related to quaternion matrices. In particular, Jia et al. designed a real structure-preserving method for quaternion Hermitian eigenvalue problems by using the structure-preserving tridiagonalization of the real counterpart for  quaternion Hermitian matrices \cite{jwl13}. Li et al. presented a structure-preserving algorithm  for computing the QSVD, which uses  the structure-preserving bidiagonalization of the real counterpart of  quaternion matrices via  Householder-based transformations  \cite{lw14, lwzz16}. Ma et al.  proposed a structure-preserving Jacobi algorithm for quaternion Hermitian eigenvalue problems \cite{mjb18}. In this paper, the columns  of a rectangle quaternion matrix  are orthogonalized in pairs by using a sequence of orthogonal JRS-symplectic Jacobi matrices to its real counterpart.  When the updated  quaternion matrix has sufficiently orthogonal columns, the SVD is obtained by column scaling.  Numerical experiments show that our method is more efficient  than the implicit Jacobi algorithm in \cite{bisa07} and gives almost the same singular values as the solver in quaternion toolbox for {\tt MATLAB} \cite{sbqtfm}.

The rest of this paper is organized as follows. In section \ref{sec2} we give necessary preliminaries used in this paper. In section \ref{sec3} a structure-preserving one-sided cyclic Jacobi algorithm is provided for computing the QSVD and  the quadratic convergence is established under some assumptions. In section \ref{sec4} we report some numerical experiments to indicate the efficiency of our algorithm and apply it to color image compression.  Finally, some concluding remarks are given in section \ref{sec5}.

\section{Preliminaries}\label{sec2}
In this section, we briefly review some necessary definitions and properties of quaternions and quaternion matrices. Quaternions were originally introduced by Hamilton in 1843 \cite{hamilton44}. For more information on quaternions, one may refer to \cite{ks89,zhangf97}  and references therein.

Throughout this paper, we need the following notation.  Let $\R$ and $\Rmn$ be the set of all real numbers and the set of all $m\times n$ real matrices respectively. Let $\mathbb{H}$ and $\mathbb{H}^{m\times n}$ be the set of all quaternions and the set of all $m\times n$ quaternion matrices respectively.  Let $I_n$ be the identity matrix of order $n$. Let $A^T, \bar{A}$ and $A^*$ stand for the transpose, conjugate and conjugate transpose of a matrix $A$ accordingly. $\|\cdot\|_F$ means the Frobenius matrix norm. For any $n\times n$ square matrix $A=(a_{pq})$, we define
\[
\off(A):=\sqrt{\sum\limits_{p=1}^n\sum\limits_{\substack{q=1 \\ q\neq p }}^n |a_{pq}|^2},
\]
where $|a_{pq}|$ denotes the absolute value of $a_{pq}$.

A quaternion $a\in\mathbb{H}$  takes the form of
\[
a=a_0+a_1i+a_2j+a_3k,
\]
where $a_0, a_1, a_2, a_3\in \mathbb{R}$ and the quaternion units $i,j,k$
satisfy the following rules
\[
i^2 = j^2 = k^2 = -1,\quad
jk=-kj=i,\quad ki=-ik=j,\quad ij=-ji=k.
\]
The conjugate of $a\in\mathbb{H}$ is given by $\bar{a}=a^*=a_0-a_1i-a_2j-a_3k$. For two quaternions
$a=a_0+a_1i+a_2j+a_3k\in\mathbb{H}$ and $b=b_0+b_1i+b_2j+b_3k\in\mathbb{H}$, their product (i.e., the Hamilton product) is given by
\begin{eqnarray*}
ab &=& a_0b_0-a_1b_1-a_2b_2-a_3b_3 \\
&& +(a_0b_1+a_1b_0+a_2b_3-a_3b_2)i \\
&& +(a_0b_2-a_1b_3+a_2b_0+a_3b_1)j \\
&& +(a_0b_3+a_1b_2-a_2b_1+a_3b_0)k.
\end{eqnarray*}
The absolute value $|a|$ of $a=a_0+a_1i+a_2j+a_3k\in\mathbb{H}$ is defined by
\[
|a| = \sqrt{a_0^2+a_1^2+a_2^2+a_3^2}=\sqrt{a\bar{a}}=\sqrt{\bar{a}a}.
\]
The multiplicative inverse of any nonzero quaternion $0\neq a\in\mathbb{H}$ is given by $a^{-1}=\bar{a}/|a|^2$.
The Hamilton product is not commutative but associative and thus $\mathbb{H}$ is an associative  division algebra over $\mathbb{R}$.

Suppose $A=A_0+A_1i+A_2j+A_3k\in\mathbb{H}^{m\times n}$ is a quaternion matrix, where $A_0, A_1, A_2, A_3\in \mathbb{R}^{m\times n}$. A real counterpart of $A$ is defined by
\begin{equation}\label{A:rc}
\Gamma_A =
\left[
\begin{array}{rrrr}
A_0 & A_2 & A_1 & A_3 \\
-A_2 & A_0 & A_3 & -A_1 \\
-A_1 & -A_3 & A_0 & A_2 \\
-A_3 & A_1 & -A_2 & A_0 \\
\end{array}
\right].
\end{equation}

Next, we recall the definitions of JRS-symmetry and JRS-symplecticity \cite{jwl13}. Let
\[
J_n=\left[
\begin{array}{rrrr}
0 & 0 & -I_n & 0 \\
0 & 0 & 0&-I_n  \\
I_n& 0 & 0 & 0 \\
0 &I_n & 0&  0
\end{array}
\right],\;
R_n=\left[
\begin{array}{rrrr}
0&-I_n & 0 & 0 \\
I_n & 0 & 0& 0 \\
0 & 0 & 0 & I_n \\
0 & 0& -I_n& 0
\end{array}
\right],\;
S_n=\left[
\begin{array}{rrrr}
0& 0& 0 & -I_n \\
0 & 0 &I_n  & 0 \\
0 & -I_n & 0 &0 \\
I_n& 0&0& 0
\end{array}
\right].
\]
A matrix $\Omega\in \mathbb{R}^{4n\times4n}$ is called JRS-symmetric if $J_n\Omega J_n^T=\Omega$, $R_n\Omega R_n^T=\Omega$ and $S_n\Omega S_n^T=\Omega$.
A matrix $\Omega\in \mathbb{R}^{4n\times4n}$ is called JRS-symplectic if $\Omega J_n\Omega^T=J_n, \Omega R_n\Omega^T=R_n$ and $\Omega S_n\Omega^T=S_n$. A matrix $\Omega\in\mathbb{R}^{4n\times4n}$ is called orthogonal JRS-symplectic if it is orthogonal and JRS-symplectic.

In the rest of this section, we recall some basic results on the relationship between a quaternion matrix and its real counterpart. First, we have the following  properties of the real counterparts of quaternion matrices \cite{jwl13,j05,wwf08}.
\begin{lemma}\label{lem21}
Let $F,G\in\mathbb{H}^{m\times n}$, $H\in\mathbb{H}^{n\times s}$,  $W\in\mathbb{H}^{n\times n}$, $\alpha\in\mathbb{R}$. Then
	\begin{itemize}
		\item[{\rm (1)}] $\Gamma_{F+G}=\Gamma_F+\Gamma_G$; $\Gamma_{\alpha G}=\alpha
		\Gamma_G$; $\Gamma_{GH}=\Gamma_G\Gamma_H$.
        \item[{\rm (2)}] $\Gamma_{G^*}=\Gamma_G^T$.
        \item[{\rm (3)}] $\Gamma_W$ is JRS-symmetric.
		\item[{\rm (4)}] $W$ is unitary if $\Gamma_W$ is orthogonal.
		\item[{\rm (5)}] If $\Gamma_G$ is orthogonal, then it is also orthogonal JRS-symplectic.
	\end{itemize}
\end{lemma}

On the SVD of a quaternion matrix, we have the following result \cite[Theorem 7.2]{zhangf97}.
\begin{lemma}\label{lem22}
Let $A\in\mathbb{H}^{m\times n}$ be a quaternion matrix with $\rank(A)=r$. Then there exist unitary quaternion matrices $U\in\mathbb{H}^{m\times m}$  and $V\in\mathbb{H}^{n\times n}$ such that
	\begin{equation}
	U^*AV=\left[
	\begin{array}{rr}
	\Sigma_r& 0 \\
	0 & 0\\
	\end{array}
	\right],
	\end{equation}
where $\Sigma_r=\diag (\sigma_1, \sigma_2, \ldots, \sigma_r)$ and $\{\sigma_w\}_{w=1}^r$ are the positive singular values of $A$.
\end{lemma}

Finally, we have the following result on the equivalence between the eigenvalue problem of a quaternion matrix and the eigenvalue problem of its real counterpart \cite{jwl13}.
\begin{lemma}\label{lem23}
Let $A=X+Yj$ be a quaternion matrix, where $X=A_0+A_1i$ and $Y=A_2+A_3i$ with $A_0, A_1, A_2, A_3\in \mathbb{R}^{n\times n}$. Then there exists a unitary quaternion
	matrix
	\[
	Q=\frac{1}{2}\left[
	\begin{array}{cccc}
	I_n & -jI_n & -iI_n & -kI_n \\
	I_n & jI_n & -iI_n & kI_n \\
	I_n & -jI_n & iI_n & kI_n \\
	I_n & jI_n & iI_n & -kI_n \\
	\end{array}
	\right]
	\]
	such that
	\begin{equation}\label{Nequa2.2}
	\Gamma_A=Q^*\left[
	\begin{array}{cccc}
	X+Yj & 0 & 0 & 0 \\
	0 & X-Yj & 0 & 0 \\
	0 & 0 & \bar{X}+\bar{Y}j & 0 \\
	0 & 0 & 0 & \bar{X}-\bar{Y}j \\
	\end{array}
	\right]Q.
	\end{equation}

\end{lemma}

\section{Structure-preserving one-sided cyclic Jacobi algorithm }\label{sec3}
In this section, we present a structure-preserving one-sided cyclic Jacobi algorithm for computing the SVD of a quaternion matrix $A=A_0+A_1i+A_2j+A_3k\in\mathbb{H}^{m\times n}$,  where $A_0, A_1, A_2, A_3\in \mathbb{R}^{m\times n}$. The proposed structure-preserving one-sided cyclic Jacobi algorithm involves a sequence of orthogonal JRS-symplectic transformations $\Gamma_A\gets \Gamma_A\Gamma_G$ such that the updated $\Gamma_A$ is closer to a column-orthogonal matrix than its predecessor. When the updated $\Gamma_A$ has sufficiently orthogonal columns, the column scaling of the updated $A$ leads to the SVD of $A$.

For simplicity, we assume that $m\ge n$. The real counterpart $\Gamma_A$ of $A$ is defined by $(\ref{A:rc})$. A one-sided cyclic Jacobi algorithm includes (a) choosing an index pair $(p,q)$ such that $1\le p<q\le n$, (b) computing a cosine-sine group $(c_r,s_0,s_1,s_2,s_3)$ such that
\BE\label{def:rj}
G(p,q,\theta)=I_n+[ \bfe_p, \bfe_q]
    \left[
          \begin{array}{cc}
            c_r-1 & s \\
            -\bar{s} & c_r-1 \\
          \end{array}
        \right]\left[
                 \begin{array}{c}
                   \bfe_p^T \\
                   \bfe_q^T \\
                 \end{array}
               \right]\in\mathbb{H}^{n\times n}
\EE
is a unitary quaternion matrix and the $p$-th and $q$-th columns of $AG(p,q,\theta)$ are orthogonal,  where $\bfe_t$ is the $t$-th unit vector and $s=s_0+s_1i+s_2j+s_3k\in\mathbb{H}$ with $c_r^2+|s|^2=1$ (In fact, this corresponds zeroing the $(p,q)$ and $(q,p)$ entries of $A^*A$ by using  $G(p,q,\theta)^*A^*AG(p,q,\theta)$), and (c) overwriting $A$ with $AG(p,q,\theta)$.

Our structure-preserving one-sided cyclic Jacobi algorithm aims to determine a sequence of  orthogonal JRS-symplectic Jacobi matrices $\{\Gamma_{G^{(\ell)}}\in\mathbb{R}^{4n\times 4n}\}_{\ell=1}^\eta$ such that  $\Gamma_{\widetilde{A}}=\Gamma_A\Gamma_{G^{(1)}}\Gamma_{G^{(2)}}\cdots$ $\Gamma_{G^{(\eta)}}$ has sufficiently orthogonal columns, which corresponds the off-diagonal entries of $\Gamma_{\widetilde{A}}^T\Gamma_{\widetilde{A}}=\Gamma_{G^{(\eta)}}^T\cdots\Gamma_{G^{(2)}}^T\Gamma_{G^{(1)}}^T\Gamma_A^T\Gamma_A\Gamma_{G^{(1)}}\Gamma_{G^{(2)}}\cdots\Gamma_{G^{(\eta)}}$ are sufficiently close to zeros. Then, by extracting the first row partitions of $ \Gamma_{\widetilde{A}}$, i.e.,
\[
[\widetilde{A}_0,  \widetilde{A}_2, \widetilde{A}_1, \widetilde{A}_3],\quad \widetilde{A}_w\in\mathbb{R}^{m\times n}\quad w=0,1,2,3,
\]
we get the updated quaternion matrix $\widetilde{A}=\widetilde{A}_0+\widetilde{A}_1i+\widetilde{A}_2j+\widetilde{A}_3k=AV$, where $V =G^{(1)}G^{(2)}\cdots$ $G^{(\eta)}$ is an $n\times n$ unitary quaternion matrix. Finally, the column scaling of $\widetilde{A}$ yields the SVD of $A$:
\begin{equation}\label{eq:tildeA}
AV=\widetilde{A} = U\Sigma,
\end{equation}
where $\Sigma=\diag (\sigma_1, \sigma_2, \ldots, \sigma_n)$ with $\sigma_w\ge 0$ for $w=1,\ldots,n$ and $U\in\mathbb{H}^{m\times n}$ satisfies $U^*U=I_n$.

The following theorem presents the orthogonalization of any two columns of an $m\times n$ quaternion matrix.
\begin{theorem}\label{theo:GG}
Let $A(p,q)=\left[\ba_p, \ba_q\right]$, where $\ba_w=\ba_{w0}+\ba_{w1}i+\ba_{w2}j+\ba_{w3}k\in\mathbb{H}^m$ is the $w$-th column of an $m\times n$ quaternion matrix $A$ for  $w=p,q$.
If $\ba_p^*\ba_q\neq 0$, then there exists a $2$-by-$2$  unitary  quaternion matrix given by
\[
G(p,q;\theta)=
\left
[\begin{array}{cc}
c_r & s\\
-\bar{s} & c_r
\end{array}
\right]
\]
such that $\widetilde{A}(p,q):=A(p,q)G(p,q;\theta)$ has orthogonal columns, where $c_r=\cos(\theta)\in\mathbb{R}$ and $s=s_0+s_1i+s_2j+s_3k\in\mathbb{H}$ with $c_r^2+|s|^2=1$.
\end{theorem}
\begin{proof}
Note that
\begin{eqnarray} \label{def:at}
\widetilde{A}(p,q) &=& A(p,q)G(p,q;\theta)  \nonumber\\
&=&
\left[\ba_p, \ba_q\right]
\left[\begin{array}{rr}
c_r & s\\
-\overline{s} & c_r
\end{array}  \right] \nonumber \\
&=& \left[c_r\ba_p-\ba_q\overline{s}, \ba_ps+c_r\ba_q\right] \nonumber \\
&:=& \left[\widetilde{{\bf a}}_p, \widetilde{{\bf a}}_q \right].
\end{eqnarray}
By hypothesis  $\ba_p^*\ba_q\neq 0$ and thus $|\ba_p^*\ba_q|>0$. Define $c_r\in\mathbb{R}$ and $s=s_0+s_1i+s_2j+s_3k\in\mathbb{H}$  by
\BE\label{def:g}
\left\{
\begin{array}{c}
s_0 = \frac{\sin(\theta)}{|\ba_p^*\ba_q|}a_{pq0},\quad s_1= \frac{\sin(\theta)}{|\ba_p^*\ba_q|}a_{pq1},\quad s_2=\frac{\sin(\theta)}{|\ba_p^*\ba_q|}a_{pq2}, \quad s_3=\frac{\sin(\theta)}{|\ba_p^*\ba_q|}a_{pq3},\quad \sin(\theta)=tc_r,\\[3mm]
c_r=\cos(\theta)=\frac{1}{\sqrt{1+t^2}},\quad  |s|=\frac{|t|}{\sqrt{1+t^2}},\quad
t=\left\{
\begin{array}{ll}
\frac{1}{\tau+ \sqrt{1+\tau^2}}, & \mbox{if $\tau\ge 0$} \\[2mm]
\frac{1}{\tau- \sqrt{1+\tau^2}},& \mbox{if $\tau<0$}
\end{array}
\right., \quad \tau=\frac{\ba_q^*\ba_q-\ba_p^*\ba_p}{{2|\ba_p^*\ba_q|}},
\end{array}
\right.
\EE
where  $\ba_p^*\ba_q:=a_{pq0}+a_{pq1}i+a_{pq2}j+a_{pq3}k$ with
\[
\left\{
\begin{array}{lcl}
a_{pq0}&=&\ba_{p0}^T\ba_{q0}+\ba_{p1}^T\ba_{q1}+\ba_{p2}^T\ba_{q2}+\ba_{p3}^T\ba_{q3},\\
a_{pq1}&=&\ba_{p0}^T\ba_{q1}-\ba_{p1}^T\ba_{q0}-\ba_{p2}^T\ba_{q3}+\ba_{p3}^T\ba_{q2},\\
a_{pq2}&=&\ba_{p0}^T\ba_{q2}+\ba_{p1}^T\ba_{q3}-\ba_{p2}^T\ba_{q0}-\ba_{p3}^T\ba_{q1},\\
a_{pq3}&=&\ba_{p0}^T\ba_{q_3}-\ba_{p1}^T\ba_{q2}+\ba_{p2}^T\ba_{q1}-\ba_{p3}^T\ba_{q0}.
\end{array}
\right.
\]
It is easy to see that $G(p,q;\theta)$ is unitary and
\[
s =\frac{tc_r}{|\ba_p^*\ba_q|}\times \ba_p^*\ba_q\quad\mbox{and}\quad
t^2+2\tau t-1=0.
\]
Thus,
\begin{eqnarray}\label{apq:orth}
\widetilde{\ba}_p^*\widetilde{{\bf a}}_q&=& (c_r\ba_p-\ba_q\overline{s})^*(\ba_ps+c_r\ba_q) \nonumber\\
&=& c_r\ba_p^*\ba_ps +c_r^2\ba_p^*\ba_q-s\ba_q^*\ba_ps-c_r\ba_q^*\ba_qs \nonumber\\
&=&c_r\times \ba_p^*\ba_p\times \frac{tc_r}{|\ba_p^*\ba_q|}\times \ba_p^*\ba_q+c_r^2\times \ba_p^*\ba_q-c_r\times \ba_q^*\ba_q\times\frac{tc_r}{|\ba_p^*\ba_q|}\times \ba_p^*\ba_q \nonumber\\
&&-\frac{tc_r}{|\ba_p^*\ba_q|}\times \ba_p^*\ba_q\times  \ba_q^*\ba_p \times\frac{tc_r}{|\ba_p^*\ba_q|}\times \ba_p^*\ba_q  \nonumber\\
&=&-c_r^2\times \ba_p^*\ba_q\times\big(t^2+\frac{\ba_q^*\ba_q-\ba_p^*\ba_p}{|\ba_p^*\ba_q|}t-1\big) \nonumber\\
&=&-c_r^2\times \ba_p^*\ba_q\times(t^2+2\tau t-1) \nonumber\\
&=& 0.
\end{eqnarray}
\end{proof}

\begin{corollary}\label{cor:GG}
If the assumptions  in Theorem \ref{theo:GG} hold, then there exists a $2$-by-$2$  unitary  quaternion matrix
\[
G(p,q;\theta)=
\left
[\begin{array}{cc}
c_r & s\\
-\bar{s} & c_r
\end{array}
\right], \quad\mbox{ $c_r=\cos(\theta)\in\mathbb{R}$ and $s=s_0+s_1i+s_2j+s_3k\in\mathbb{H} $}
\]
such that
\[
\Gamma_{G(p,q;\theta)}^T\Gamma_{A(p,q)}^T\Gamma_{A(p,q)}\Gamma_{G(p,q;\theta)}
= \widetilde{B}(p,q;p,q)\oplus \widetilde{B}(p,q;p,q)\oplus \widetilde{B}(p,q;p,q)\oplus \widetilde{B}(p,q;p,q)
\]
for some
\[
\widetilde{B}(p,q;p,q)=
\left[
\begin{array}{cc}
b_{pp} & 0\\
0 & b_{qq}
\end{array}
\right].
\]
\end{corollary}
\begin{proof}
By Theorem \ref{theo:GG} we know that $\widetilde{A}(p,q)=A(p,q)G(p,q;\theta)$ has orthogonal columns for $G(p,q;\theta)$ with $c_r$ and $s$ defined as in (\ref{def:g}).
From (\ref{def:at}) and (\ref{apq:orth}) we have
\BE\label{def:bt}
\widetilde{B}(p,q;p,q):= \widetilde{A}^*(p,q)\widetilde{A}(p,q)
= \left[\begin{array}{rr}
\widetilde{\ba}_p^*\widetilde{{\bf a}}_p & \widetilde{\ba}_p^*\widetilde{{\bf a}}_q\\
\widetilde{\ba}_q^*\widetilde{{\bf a}}_p & \widetilde{\ba}_q^*\widetilde{{\bf a}}_q
\end{array}  \right] = \left[\begin{array}{rr}
\widetilde{\ba}_p^*\widetilde{{\bf a}}_p &0 \\
0 & \widetilde{\ba}_q^*\widetilde{{\bf a}}_q
\end{array}  \right].
\EE
This shows that $\widetilde{B}(p,q;p,q)$ is  a real diagonal matrix, where $\widetilde{\ba}_p=c_r\ba_p-\ba_q\overline{s}$ and $\widetilde{\ba}_q= \ba_ps+c_r\ba_q$ with $c_r$ and $s$ being defined in (\ref{def:g}). Using Lemma \ref{lem21} and  (\ref{def:bt}) we find
\begin{eqnarray*}
&&\Gamma_{G(p,q;\theta)}^T\Gamma_{A(p,q)}^T\Gamma_{A(p,q)}\Gamma_{G(p,q;\theta)} \\
&=&\Gamma_{\widetilde{A}^*(p,q)}\Gamma_{\widetilde{A}(p,q)} =  \Gamma_{\widetilde{A}^*(p,q)\widetilde{A}(p,q)} =\Gamma_{\widetilde{B}(p,q;p,q)}\\
&=& \widetilde{B}(p,q;p,q)\oplus \widetilde{B}(p,q;p,q)\oplus \widetilde{B}(p,q;p,q)\oplus \widetilde{B}(p,q;p,q),
\end{eqnarray*}
where $b_{pp}=\widetilde{\ba}_p^*\widetilde{{\bf a}}_p$ and  $b_{qq}=\widetilde{\ba}_q^*\widetilde{{\bf a}}_q$.
\end{proof}

\begin{remark}
For $t$ defined in (\ref{def:g}), we see that the rotation angle satisfies $|\theta|\le\pi/4$.
Also, from (\ref{def:at}) and (\ref{def:bt}) we have
\begin{eqnarray*}
\left[\begin{array}{rr}
\widetilde{\ba}_p^*\widetilde{{\bf a}}_p &0 \\
0 & \widetilde{\ba}_q^*\widetilde{{\bf a}}_q
\end{array}  \right]
&=& \widetilde{A}^*(p,q)\widetilde{A}(p,q)\\
&=& G^*(p,q;\theta)A^*(p,q)A(p,q)G(p,q;\theta)\\
&=& \left[\begin{array}{rr}
c_r & s\\
-\overline{s} & c_r
\end{array}  \right]^*\left[\begin{array}{rr}
\ba_p^*\ba_p & \ba_p^*\ba_q\\
\ba_q^*\ba_p & \ba_q^*\ba_q
\end{array}  \right]\left[\begin{array}{rr}
c_r & s\\
-\overline{s} & c_r
\end{array}  \right].
\end{eqnarray*}
Since the Frobenius norm is unitary invariant we obtain
\begin{equation}\label{equ:FroUni}
(\widetilde{\ba}_p^*\widetilde{{\bf a}}_p)^2+(\widetilde{\ba}_q^*\widetilde{{\bf a}}_q)^2 = (\ba_p^*\ba_p)^2+ 2|\ba_p^*\ba_q|^2+(\ba_q^*\ba_q)^2.
\end{equation}
\end{remark}

\begin{remark} \label{rem:offa}
We observe from Theorem \ref{theo:GG} that the matrix $\widetilde{A}^*\widetilde{A}$ agrees with $A^*A$ except in the $p$-th and $q$-th rows and the $p$-th and $q$-th columns. Let $A=\left[\ba_1, \ba_2, \ldots, \ba_n\right]$ and $\widetilde{A}=\left[\widetilde{\ba}_1, \widetilde{\ba}_2, \ldots,\widetilde{\ba}_n\right]$.
Using (\ref{equ:FroUni}) we have
\begin{eqnarray} \label{off}
\off(\widetilde{A}^*\widetilde{A})^2 &=& \|\widetilde{A}^*\widetilde{A}\|_F^2-\sum\limits_{t=1}^n (\widetilde{\ba}_t^*\widetilde{\ba}_t)^2 \nonumber\\
&=& \|A^*A\|_F^2 -\sum\limits_{t=1}^n (\ba_t^*\ba_t)^2 +\big((\ba_p^*\ba_p)^2+(\ba_q^*\ba_q)^2 -(\widetilde{\ba}_p^*\widetilde{\ba}_p)^2-(\widetilde{\ba}_q^*\widetilde{\ba}_q)^2\big) \nonumber\\
&=& \off(A^*A)^2-2|\ba_p^*\ba_q|^2.
\end{eqnarray}
Since $\|\Gamma_A^T\Gamma_A\|_F^2=4\|A^*A\|_F^2$, we know that $\off(\Gamma_A^T\Gamma_A)^2=4\off(A^*A)^2$. Using (\ref{off}), this implies
that $\Gamma_A^T\Gamma_A$ is closer to a diagonal matrix with each  orthogonal JRS-symplectic Jacobi rotation.
\end{remark}

Based on Theorem \ref{theo:GG} and Corollary \ref{cor:GG}, we present the following algorithm for generating a $2$-by-$2$   unitary  quaternion Jacobi matrix for orthogonalizing any two columns of an $m\times n$ quaternion matrix. This algorithm needs $96m+30$ operations.

\begin{algorithm}\label{alg:1}
{\rm Given $A(p,q)=
\left[\ba_p, \ba_q\right]$, where $\ba_w=\ba_{w0}+\ba_{w1}i+\ba_{w2}j+\ba_{w3}k\in\mathbb{H}^m$ is the $w$-th column of an $m\times n$ quaternion matrix $A$ for  $w=p,q$, this algorithm computes a cosine-sine group $(c_r,s_0,s_1,s_2,s_3)$ such that  $\widetilde{A}(p,q)=A(p,q)G(p,q;\theta)$ has two orthogonal columns. }
\begin{flushleft}
{\bf function}~~{\bf $(c_r,s_0,s_1,s_2,s_3)={\bf GJSJR}(\ba_{p0},\ba_{p1},\ba_{p2},\ba_{p3},\ba_{q0},\ba_{q1},\ba_{q2},\ba_{q3})$}\\
~~~~$a_{pp0}=\ba_{p0}^T\ba_{p0}+\ba_{p1}^T\ba_{p1}+\ba_{p2}^T\ba_{p2}+\ba_{p3}^T\ba_{p3}$,
$a_{pp1}=\ba_{p0}^T\ba_{p1}-\ba_{p1}^T\ba_{p0}-\ba_{p2}^T\ba_{p3}+\ba_{p3}^T\ba_{p2}$\\
~~~~$a_{pp2}=\ba_{p0}^T\ba_{p2}+\ba_{p1}^T\ba_{p3}-\ba_{p2}^T\ba_{p0}-\ba_{p3}^T\ba_{p1}$,
$a_{pp3}=\ba_{p0}^T\ba_{p3}-\ba_{p1}^T\ba_{p2}+\ba_{p2}^T\ba_{p1}-\ba_{p3}^T\ba_{p0}$\\
~~~~$a_{pp}=\sqrt{a_{pp0}^2+a_{pp1}^2+a_{pp2}^2+a_{pp3}^2}$ \\
~~~~$a_{qq0}=\ba_{q0}^T\ba_{q0}+\ba_{q1}^T\ba_{q1}+\ba_{q2}^T\ba_{q2}+\ba_{q3}^T\ba_{q3}$,
$a_{qq1}=\ba_{q0}^T\ba_{q1}-\ba_{q1}^T\ba_{q0}-\ba_{q2}^T\ba_{q3}+\ba_{q3}^T\ba_{q2}$\\
~~~~$a_{qq2}=\ba_{q0}^T\ba_{q2}+\ba_{q1}^T\ba_{q3}-\ba_{q2}^T\ba_{q0}-\ba_{q3}^T\ba_{q1}$,
$a_{qq3}=\ba_{q0}^T\ba_{q_3}-\ba_{q1}^T\ba_{q2}+\ba_{q2}^T\ba_{q1}-\ba_{q3}^T\ba_{q0}$\\
~~~~$a_{qq}=\sqrt{a_{qq0}^2+a_{qq1}^2+a_{qq2}^2+a_{qq3}^2}$ \\
~~~~$a_{pq0}=\ba_{p0}^T\ba_{q0}+\ba_{p1}^T\ba_{q1}+\ba_{p2}^T\ba_{q2}+\ba_{p3}^T\ba_{q3}$,
$a_{pq1}=\ba_{p0}^T\ba_{q1}-\ba_{p1}^T\ba_{q0}-\ba_{p2}^T\ba_{q3}+\ba_{p3}^T\ba_{q2}$\\
~~~~$a_{pq2}=\ba_{p0}^T\ba_{q2}+\ba_{p1}^T\ba_{q3}-\ba_{p2}^T\ba_{q0}-\ba_{p3}^T\ba_{q1}$,
$a_{pq3}=\ba_{p0}^T\ba_{q_3}-\ba_{p1}^T\ba_{q2}+\ba_{p2}^T\ba_{q1}-\ba_{p3}^T\ba_{q0}$\\
~~~~$a_{pq}=\sqrt{a_{pq0}^2+a_{pq1}^2+a_{pq2}^2+a_{pq3}^2}$ \\
~~~~{\bf if} $a_{pq}=0$\\
~~~~~~~$c_r=1$, $s_0=s_1=s_2=s_3=0$\\
~~~~{\bf else}\\
~~~~~~~$\tau=(a_{qq}-a_{pp})/(2a_{pq})$\\
~~~~~~~{\bf if} $\tau\ge 0$\\
~~~~~~~~~~$t=1/(\tau+\sqrt{1+\tau^{2}})$\\
~~~~~~~{\bf else}\\
~~~~~~~~~~$t=1/(\tau-\sqrt{1+\tau^{2}})$\\
~~~~~~{\bf end}\\
~~~~~~$c_r=1/\sqrt{1+t^{2}}$, $\delta=tc_r /a_{pq}$, $s_0 = \delta a_{pq0}$, $s_1= \delta a_{pq1}$, $s_2=\delta a_{pq2}$, $s_3=\delta a_{pq3}$\\
~~~~{\bf end}
\end{flushleft}
\end{algorithm}

Algorithm \ref{alg:1} gives a scheme for orthogonalizing columns $p$ and $q$ of an $m\times n$ quaternion matrix $A=\left[\ba_1, \ba_2, \ldots, \ba_n\right]$. One may choose $p$ and $q$ such that $|\ba_p^*\ba_q|$ is maximal as the classical Jacobi algorithm \cite[Algorithm 8.4.2]{gl89}.

The following algorithm describes a structure-preserving one-sided classical Jacobi algorithm, which is such that  a quaternion matrix  has sufficiently orthogonal columns.
\begin{algorithm}\label{alg:3}
{\rm Given an $m\times n$ quaternion matrix $A=A_0+A_1i+A_2j+A_3k\in\mathbb{H}^{m\times n}$ and a tolerance $\tol>0$, this algorithm overlaps the real counterpart $\Gamma_A$ by $\Gamma_A\widetilde{V}$, where $\widetilde{V}$ is  orthogonal and $\off(\widetilde{V}^T \Gamma_A^T\Gamma_A \widetilde{V})\le\tol\cdot \|\Gamma_A^T\Gamma_A\|_F$.}
\begin{flushleft}
$\widetilde{V}=\Gamma_{I_n}$, $\zeta=\tol\cdot\|\Gamma_A^T\Gamma_A\|_F$ \\
{\bf while}~ $\off(\Gamma_A^T\Gamma_A)>\zeta$\\
~~~~~~{\rm Choose $(p,q)$ so $|\ba_p^*\ba_q|=\max_{u\neq v}|\ba_u^*\ba_v|$}\\
~~~~~~~$\ba_{p0}=A_0(:,p)$, $\ba_{p1}=A_1(:,p)$, $\ba_{p2}=A_2(:,p)$, $\ba_{p3}=A_3(:,p)$\\
~~~~~~~$\ba_{q0}=A_0(:,q)$, $\ba_{q1}=A_1(:,q)$, $\ba_{q2}=A_2(:,q)$, $\ba_{q3}=A_3(:,q)$\\
~~~~~~~$(c_r,s_0,s_1,s_2,s_3)={\bf GJSJR}(\ba_{p0},\ba_{p1},\ba_{p2},\ba_{p3},\ba_{q0},\ba_{q1},\ba_{q2},\ba_{q3})$\\
~~~~~~~$\Gamma_A=\Gamma_A\Gamma_{G(p,q,\theta)}$\\
~~~~~~~$\widetilde{V} = \widetilde{V}\Gamma_{G(p,q,\theta)}$\\
{\bf end}
\end{flushleft}
\end{algorithm}

Algorithm \ref{alg:3} gives the following basic iterative scheme:
\BE\label{it:k1}
\Gamma_{A^{(\ell+1)}}=\Gamma_{A^{(\ell)}}\Gamma_{G^{(\ell)}},\quad \ell=0,1,2,\ldots,
\EE
where $A^{(0)}=A$ and $G^{(\ell)}\in\mathbb{H}^{n\times n}$ is a unitary quaternion matrix defined in (\ref{def:rj}) with the cosine-sine group $(c_r,s_0,s_1,s_2,s_3)$ being generated by Algorithm \ref{alg:1}. Algorithm \ref{alg:3} may be seen as a structure-preserving Jacobi algorithm for solving the eigenvalue problem of an $n\times n$ quaternion Hermitian matrix $B:=A^*A$ as in \cite{mjb18}:
\begin{eqnarray*}
\Gamma_{B^{(\ell+1)}}&=&\Gamma_{A^{(\ell+1)}}^T\Gamma_{A^{(\ell+1)}}= \Gamma_{G^{(\ell)}}^T\Gamma_{A^{(\ell)}}^T\Gamma_{A^{(\ell)}}\Gamma_{G^{(\ell)}}\\
&=& \Gamma_{G^{(\ell)}}^T\Gamma_{(A^{(\ell)})^*A^{(\ell)}}\Gamma_{G^{(\ell)}} =
\Gamma_{G^{(\ell)}}^T\Gamma_{B^{(\ell)})}\Gamma_{G^{(\ell)}},\quad \ell=0,1,2,\ldots,
\end{eqnarray*}
where  $B^{(0)}=A^*A$.
Let
\[
A^{(\ell)}:=[\ba_1^{(\ell)},\ldots, \ba_p^{(\ell)},\ldots,\ba_q^{(\ell)},\ldots,\ba_n^{(\ell)}].
\]
Then
\begin{eqnarray}\label{def:bl}
B^{(\ell)} &:=& (A^{(\ell)})^*A^{(\ell)}\nonumber \\
&=&  \left[
\begin{array}{ccccccc}
(\ba_1^{(\ell)})^*\ba_1^{(\ell)} & \cdots & (\ba_1^{(\ell)})^*\ba_p^{(\ell)} & \cdots &(\ba_1^{(\ell)})^*\ba_q^{(\ell)} &\cdots & (\ba_1^{(\ell)})^*\ba_n^{(\ell)}  \\
\vdots & \ddots & \vdots & \ddots &\vdots &\ddots & \vdots  \\
(\ba_p^{(\ell)})^*\ba_1^{(\ell)} & \cdots & (\ba_p^{(\ell)})^*\ba_p^{(\ell)} & \cdots &(\ba_p^{(\ell)})^*\ba_q^{(\ell)} &\cdots & (\ba_p^{(\ell)})^*\ba_n^{(\ell)}  \\
\vdots & \ddots & \vdots & \ddots &\vdots &\ddots & \vdots  \\
(\ba_q^{(\ell)})^*\ba_1^{(\ell)} & \cdots & (\ba_q^{(\ell)})^*\ba_p^{(\ell)} & \cdots &(\ba_q^{(\ell)})^*\ba_q^{(\ell)} &\cdots & (\ba_q^{(\ell)})^*\ba_n^{(\ell)}  \\
\vdots & \ddots & \vdots & \ddots &\vdots &\ddots & \vdots  \\
(\ba_n^{(\ell)})^*\ba_1^{(\ell)} & \cdots & (\ba_n^{(\ell)})^*\ba_p^{(\ell)} & \cdots &(\ba_n^{(\ell)})^*\ba_q^{(\ell)} &\cdots & (\ba_n^{(\ell)})^*\ba_n^{(\ell)}
\end{array}
\right].
\end{eqnarray}
From Remark \ref{rem:offa} we have
\begin{eqnarray*}
&&\off(B^{(\ell)})^2 = \|B^{(\ell)}\|_{F}^{2}- \sum_{u=1}^{n}{({b}^{(\ell)}_{uu}})^2 \nonumber \\
&=&  \|(A^{(\ell)})^*A^{(\ell)}\|_{F}^{2}- \sum_{u=1}^{n}{({b}^{(\ell)}_{uu}})^2 \nonumber \\
&=&\|{(G^{(\ell-1)})}^*B^{(\ell-1)}G^{(\ell-1)}\|_{F}^{2}
-\Big(\sum_{u=1}^{n}({b^{{(\ell-1)}}_{uu}})^2-(b_{pp}^{(\ell-1)})^2-(b_{qq}^{(\ell-1)})^{2}+(b_{pp}^{(\ell)})^2+(b_{qq}^{(\ell)})^{2}\Big)\nonumber\\
&=& \|B^{(\ell-1)}\|_{F}^{2} -\sum_{u=1}^{n}({b^{{(\ell-1)}}_{uu}})^2+\big((b_{pp}^{(\ell-1)})^2+(b_{qq}^{(\ell-1)})^{2}-(b_{pp}^{(\ell)})^2-(b_{qq}^{(\ell)})^{2}\big) \nonumber\\
&=&\texttt{off}(B^{(\ell-1)})^2-2|b_{pq}^{(\ell-1)}|^{2}.
\end{eqnarray*}
We see that $\|\Gamma_{B^{(\ell)}}\|_F^2=4\|B^{(\ell)}\|_F^2$. Hence, $\texttt{off}(\Gamma_{B^{(\ell)}})^2=4 \texttt{off}({B^{(\ell)}})^2$.

We have the following result on the linear convergence of Algorithm \ref{alg:3}. The proof follows from \cite[Theorem 3.3]{mjb18} and thus we omit it here.
\begin{theorem}\label{thm:lc}
Let $\{\sigma_w\}_{w=1}^n$ be the $n$ singular values of $A$ and $\Gamma_{A^{(\ell)}}$ be the matrix after $\ell$ orthogonal JRS-symplectic Jacobi updates generated by Algorithm \ref{alg:3}. Then there exists a permutation $\{\varphi_1,\varphi_2,\ldots,\varphi_n\}$ of $\{1,2,\ldots,n\}$ such that
\[
\lim_{\ell\to\infty}\Gamma_{A^{(\ell)}}^T\Gamma_{A^{(\ell)}}=\diag(\sigma_{\varphi_1}^2,\ldots,\sigma_{\varphi_n}^2,\sigma_{\varphi_1}^2,\ldots,\sigma_{\varphi_n}^2,
\sigma_{\varphi_1}^2,\ldots,\sigma_{\varphi_n}^2,\sigma_{\varphi_1}^2,\ldots,\sigma_{\varphi_n}^2).
\]
Moreover,
\[
\off(\Gamma^T_{A^{(\ell)}}\Gamma_{A^{(\ell)}})\le \Big(1-\frac{1}{N}\Big)^\ell \off(\Gamma^T_{A^{(0)}}\Gamma_{A^{(0)}}),\quad N:=\frac{1}{2}n(n-1).
\]
\end{theorem}

However, in  Algorithm \ref{alg:3}, the search for the optimal columns $p$ and $q$ needs $O(n^2)$. To reduce the cost, one may adopt the scheme of  cyclic-by-column as the cyclic Jacobi algorithm \cite[Algorithm 8.4.3]{gl89}. In the following procedure, we provide a structure-preserving one-sided cyclic Jacobi algorithm for orthogonalizing  the columns of an $m\times n$ quaternion matrix $A$.
\begin{algorithm}\label{alg:2}
{\rm Given an $m\times n$ quaternion matrix $A=A_0+A_1i+A_2j+A_3k\in\mathbb{H}^{m\times n}$ and a tolerance $\tol>0$, this algorithm overlaps the real counterpart $\Gamma_A$ by $\Gamma_A\widetilde{V}$, where $\widetilde{V}$ is  orthogonal and $\off(\widetilde{V}^T \Gamma_A^T\Gamma_A \widetilde{V})\le\tol\cdot \|\Gamma_A^T\Gamma_A\|_F$.}

\begin{flushleft}
$\widetilde{V}=\Gamma_{I_n}$, $\zeta=\tol\cdot\|\Gamma_A^T\Gamma_A\|_F$ \\
{\bf while}~ $\off(\Gamma_A^T\Gamma_A)>\zeta$\\
~~~{\bf for} $p=1:n-1$\\
~~~~~~~~{\bf for} $q=p+1:n$\\
~~~~~~~~~~~~~$\ba_{p0}=A_0(:,p)$, $\ba_{p1}=A_1(:,p)$, $\ba_{p2}=A_2(:,p)$, $\ba_{p3}=A_3(:,p)$\\
~~~~~~~~~~~~~$\ba_{q0}=A_0(:,q)$, $\ba_{q1}=A_1(:,q)$, $\ba_{q2}=A_2(:,q)$, $\ba_{q3}=A_3(:,q)$\\
~~~~~~~~~~~~~$(c_r,s_0,s_1,s_2,s_3)={\bf GJSJR}(\ba_{p0},\ba_{p1},\ba_{p2},\ba_{p3},\ba_{q0},\ba_{q1},\ba_{q2},\ba_{q3})$\\
~~~~~~~~~~~~~$\Gamma_A=\Gamma_A\Gamma_{G(p,q,\theta)}$\\
~~~~~~~~~~~~~$\widetilde{V} = \widetilde{V}\Gamma_{G(p,q,\theta)}$\\
~~~~~~~~{\bf end}\\
~~~{\bf end}\\
{\bf end}
\end{flushleft}
\end{algorithm}

\begin{remark}
In Algorithm \ref{alg:2}, we only need to store the first $m$ rows of $\Gamma_A$, which reduce the total storage. The later numerical tests show that Algorithm \ref{alg:2} works much better than the implicit Jacobi algorithm in {\rm \cite{bisa07}}.
\end{remark}

We now give the quadratic convergence analysis of Algorithm \ref{alg:2}.  Algorithm \ref{alg:2} gives the following iterative scheme:
\[
\Gamma_{A^{(\ell+1)}}=\Gamma_{A^{(\ell)}}\Gamma_{G^{(\ell)}},\quad \ell=0,1,2,\ldots,
\]
where $A^{(0)}=A$ and $G^{(\ell)}\in\mathbb{H}^{n\times n}$ is a unitary quaternion matrix defined in (\ref{def:rj}) with the cosine-sine group $(c_r,s_0,s_1,s_2,s_3)$ being generated by Algorithm \ref{alg:1}. In fact, Algorithm \ref{alg:2} can be seen as a  structure-preserving cyclic Jacobi algorithm for $\Gamma_A^T\Gamma_A$.

We have the following theorem on the quadratic convergence of Algorithm \ref{alg:2}. The proof can be seen as a generalization of \cite{vk66,w62}.
\begin{theorem}\label{thm:qc}
Let $\{\sigma_w\}_{w=1}^n$ be the $n$ singular values of $A$ and $\Gamma_{A^{(\ell)}}$ be the matrix after $\ell$ orthogonal JRS-symplectic Jacobi updates generated by Algorithm \ref{alg:2}.  If $\off(\Gamma^T_{A^{(d)}}\Gamma_{A^{(d)}})<\delta/2$ for some $d\ge 1$ where $ 0<2\delta\le\min_{\sigma_u\neq \sigma_v}|\sigma_u^2-\sigma_v^2|$,  then
\[
\off(\Gamma^T_{A^{(d+N)}}\Gamma_{A^{(d+N)}})\le \sqrt{\frac{25}{72}}\cdot \frac{\off(\Gamma^T_{A^{(d)}}\Gamma_{A^{(d)}})^2}{\delta}.
\]
\end{theorem}
\begin{proof}
Write  $S^{(d)}:=\Gamma_{A^{(d)}}^T\Gamma_{A^{(d)}}=D^{(d)}+E^{(d)}+(E^{(d)})^T$, where $D^{(d)}$ and $E^{(d)}$ are diagonal and strictly upper triangular, respectively. By using Theorem \ref{thm:lc} and the Wielandt-Hoffman theorem (\cite[Theorem 8.1.4]{gl89}) we have
\begin{equation}
|s_{ww}^{(d)}-\sigma_{\varphi_w}^2|\leq \| D^{(d)}-S^{(d)}\|_F< \frac{\delta}{2},\quad 1\le w\le n.
\end{equation}
Thus we have for two distinct eigenvalues $\sigma_{\varphi_u}^2$ and $\sigma_{\varphi_v}^2$,
\begin{eqnarray}\label{eq:diaggap}
|s_{uu}^{(d)}-s_{vv}^{(d)}|&=&|(s_{uu}^{(d)}-\sigma_{\varphi_u}^2)-(s_{vv}^{(d)}-\sigma_{\varphi_v}^2)+(\sigma_{\varphi_u}^2-\sigma_{\varphi_v}^2)| \nonumber\\
&\geq&|\sigma_{\varphi_u}^2-\sigma_{\varphi_v}^2|-|s_{uu}^{(d)}-\sigma_{\varphi_u}^2|-|s_{vv}^{(d)}-\sigma_{\varphi_v}^2| \nonumber\\
&>&2\delta-\frac{\delta}{2}-\frac{\delta}{2}= \delta.
\end{eqnarray}
Since $\off(\Gamma^T_{A^{(\ell)}}\Gamma_{A^{(\ell)}})$ is decreasing, we know that $\off(\Gamma^T_{A^{(\ell)}}\Gamma_{A^{(\ell)}})<\delta/2$ and (\ref{eq:diaggap})  hold for $\ell> d$.

We show the quadratic convergence of Algorithm \ref{alg:2}. We first consider the case of one  multiple singular value.
Assume without loss of generality that only $\sigma_{\varphi_1}$ is a multiple singular value of $A^{(\ell)}$ $(\ell>d)$ and the diagonal entries $s_{11}^{(\ell)}, s_{22}^{(\ell)}, \ldots, s_{n_1n_1}^{(\ell)}$ of $S^{(\ell)}:=\Gamma^T_{A^{(\ell)}}\Gamma_{A^{(\ell)}}$ converge to $\sigma_{\varphi_1}^2$. Then, by appropriate row and column interchanges, we get a permutation matrix $P$ such that
\[
\widehat{S}^{(\ell)}=P^TS^{(\ell)}P=\left[
\begin{array}{cc}
\widehat{S}_{11}^{(\ell)} & \widehat{S}_{12}^{(\ell)} \\
\widehat{S}_{21}^{(\ell)}  & \widehat{S}_{22}^{(\ell)}
\end{array}
\right],
\]
where $\widehat{S}_{11}^{(\ell)}\in\mathbb{R}^{4n_1\times 4n_1}$ with diagonal entries converging to  $\sigma_{\varphi_1}^2$.

We provide an upper bound for the quantity
\[
\Phi_1^{(\ell)}:=\sqrt{\sum_{1\le p\neq q\le 4n_1}(\hat{s}_{pq}^{(\ell)})^2}.
\]
As in \cite{w65}, it is easy to see that
\begin{eqnarray}
\nonumber T^{(\ell)}&=&\left[
\begin{array}{cc}
I_{4n_1} & -\widehat{S}_{12}^{(\ell)}(\widehat{S}_{22}^{(\ell)} -\sigma_{\varphi_1}^2I_{4n-4n_1})^{-1}\\
0  & I_{4n-4n_1}
\end{array}
\right]\left[
\begin{array}{cc}
\widehat{S}_{11}^{(\ell)}-\sigma_{\varphi_1}^2I_{4n_1} & \widehat{S}_{12}^{(\ell)} \\
\widehat{S}_{21}^{(\ell)}  & \widehat{S}_{22}^{(\ell)}- \sigma_{\varphi_1}^2I_{4n-4n_1}
\end{array}
\right]\\
\nonumber&=&\left[
\begin{array}{cc}
\widehat{S}_{11}^{(\ell)}-\sigma_{\varphi_1}^2I_{4n_1}-\widehat{S}_{12}^{(\ell)}(\widehat{S}_{22}^{(\ell)} -\sigma_{\varphi_1}^2I_{4n-4n_1})^{-1}\widehat{S}_{21}^{(\ell)} & 0 \\
\widehat{S}_{21}^{(\ell)}  & \widehat{S}_{22}^{(\ell)}- \sigma_{\varphi_1}^2I_{4n-4n_1}
\end{array}
\right]
\end{eqnarray}
and  the rank of $T^{(\ell)}$ is the same as $\widehat{S}^{(\ell)}-\sigma_{\varphi_1}^2I_{4n}$, which implies that
\BE\label{eq:s1sig}
\widehat{S}_{11}^{(\ell)}-\sigma_{\varphi_1}^2I_{4n_1}=\widehat{S}_{12}^{(\ell)}(\widehat{S}_{22}^{(\ell)} -\sigma_{\varphi_1}^2I_{4n-4n_1})^{-1}\widehat{S}_{21}^{(\ell)}.
\EE
Let $\widetilde{\sigma}_{\varphi_w}^2$ be  the eigenvalues  of $\widehat{S}_{22}^{(\ell)}$. Note that
\[
|\sigma_{\varphi_w}^2-\widetilde{\sigma}_{\varphi_w}^2|\leq \left\|\widehat{S}^{(\ell)}-
\left[
\begin{array}{cc}
\Sigma_{11}^{(\ell)} & 0 \\
0  & \widehat{S}_{22}^{(\ell)}
\end{array}
\right]\right\|_F \leq \off(\widehat{S}^{(\ell)})= \off(S^{(\ell)}) \leq \frac{\delta}{2},
\]
where $\Sigma_{11}^{(\ell)}=\diag( \hat{s}_{11}^{(\ell)},\hat{s}_{22}^{(\ell)}, \ldots, \hat{s}_{4n_1,4n_1}^{(\ell)})$ with $\hat{s}_{ww}^{(\ell)}$ being the $(w,w)$ entry of $\widehat{S}_{11}^{(\ell)}$.
Thus,
\[
|\sigma_{\varphi_1}^2-\widetilde{\sigma}_{\varphi_w}^2|\geq |\sigma_{\varphi_1}^2-\sigma_{\varphi_w}^2|-|\sigma_{\varphi_w}^2-\widetilde{\sigma}_{\varphi_w}^2|\geq 2\delta-\frac{\delta}{2}=\frac{3\delta}{2}.
\]
This, together with (\ref{eq:s1sig}), yields
\begin{subequations}\label{eq:5}
\begin{eqnarray}
(\Phi_1^{(\ell)})^2&=&\off(\widehat{S}_{11}^{(\ell)})^2\le\|\widehat{S}_{11}^{(\ell)}-\sigma_{\varphi_1}^2I_{4n_1}\|_F^2\leq\frac{\|\widehat{S}_{12}^{(\ell)}\|_F^4}{\min_{\widetilde{\sigma}_{\varphi_w}^2\neq \sigma_{\varphi_1}^2}|\sigma_{\varphi_1}^2-\widetilde{\sigma}_{\varphi_w}^2|^2} \nonumber\\
&\leq&  \frac{2\off(S^{(\ell)})^2}{9\delta^2}\|\widehat{S}_{12}^{(\ell)}\|_F^2
\le  \frac{2\off(S^{(d)})^2}{9\delta^2}\|\widehat{S}_{12}^{(\ell)}\|_F^2 \\
&\leq& \frac{\off(S^{(d)})^4}{9\delta^2}
\end{eqnarray}
\end{subequations}
since  $\|\widehat{S}_{12}^{(\ell)}\|_F^2\le 1/2\; \off(\widehat{S}^{(\ell)})^2=1/2\; \off(S^{(\ell)})^2\le 1/2 \off(S^{(d)})^2$. The estimates in (\ref{eq:5}) is crucial  for proving the quadratic convergence of the structure-preserving one-sided cyclic Jacobi algorithm.

As in (\ref{def:g}), the rotation angle $\theta_\ell$ is chosen such that $|\theta_\ell|\leq\pi/4$. Using  (\ref{def:bl}), (\ref{eq:diaggap}) and $S^{(\ell)}=\Gamma_{B^{(\ell)}}$ we have
\begin{equation}\label{eq:2.1}
|\sin\theta_\ell|  \leq  \frac{1}{2}|\tan 2\theta_\ell|=\frac{|b_{pq}^{(\ell-1)}|}{|b_{qq}^{(\ell-1)}-b_{pp}^{(\ell-1)}|} \leq \frac{|b_{pq}^{(\ell-1)}|}{\delta}.
\end{equation}
This, together with $\frac{1}{2}\off(B^{(\ell-1)})^2-\frac{1}{2}\off(B^{(\ell)})^2=|b_{pq}^{(\ell-1)}|^2$, yields
\[
\off (B^{(d)})^2-2\sum_{\ell=d+1}^{d+N}|b_{pq}^{(\ell-1)}|^2=\off (B^{(d+N)})^2\geq 0.
\]
Using  (\ref{eq:2.1}) we have
\begin{equation}\label{eq:6}
\widehat{\sum}\sin^2\theta_\ell\leq \frac{\sum |b_{pq}^{(\ell-1)}|^2}{\delta^2}\leq \frac{\off (B^{(d)})^2}{2\delta^2}.
\end{equation}
where $\widehat{\sum}$ means that we include in the sum only rotations of entries outside the first $n_1$ rows and the first $n_1$ columns of $B^{(\ell)}$.

We now show the quadratic convergence of Algorithm \ref{alg:2} for the case of one  multiple singular value.
As in \cite{w62}, for example, we take an $m\times 5$ quaternion matrix $A$. In this case, $B^{(0)}=A^{*}A\in \mathbb{H}^{5\times 5}$. In the following, we show the effect of annihilating the entries in the first row and column of $B^{(d)}$. Since we are only interested in the off diagonal entries, which is updated when these entries are affected by the current rotations, the diagonal entries are all denoted by $``\times"$.

\begin{eqnarray}\label{eq:5m}
B^{(d)} &:=& \left[
\begin{array}{ccccc}
\times & b_{12}^{(d)} & b_{13}^{(d)} & b_{14}^{(d)} & b_{15}^{(d)} \\
b_{21}^{(d)} & \times & b_{23}^{(d)} & b_{24}^{(d)} & b_{25}^{(d)} \\
b_{31}^{(d)} & b_{32}^{(d)} & \times & b_{34}^{(d)} & b_{35}^{(d)} \\
b_{41}^{(d)} & b_{42}^{(d)} & b_{43}^{(d)} & \times & b_{45}^{(d)}\\
b_{51}^{(d)} & b_{52}^{(d)} & b_{53}^{(d)} & b_{54}^{(d)} & \times\\
\end{array}
\right]
\xrightarrow[]{G(1,2; \theta_d)}
\left[
\begin{array}{ccccc}
\times & 0 & b_{13}^{(d+1)} & b_{14}^{(d+1)} & b_{15}^{(d+1)} \\
0 & \times & b_{23}^{(d+1)} & b_{24}^{(d+1)} & b_{25}^{(d+1)} \\
b_{31}^{(d+1)} & b_{32}^{(d+1)} & \times & b_{34}^{(d)} & b_{35}^{(d)} \\
b_{41}^{(d+1)} & b_{42}^{(d+1)} & b_{43}^{(d)} & \times & b_{45}^{(d)} \\
b_{51}^{(d+1)} & b_{52}^{(d+1)} & b_{53}^{(d)} & b_{54}^{(d)} & \times \\
\end{array}
\right] \nonumber\\
&& \xrightarrow[]{G(1,3; \theta_{d+1})}
\left[
\begin{array}{ccccc}
\times & b_{12}^{(d+2)} & 0 & b_{14}^{(d+2)} & b_{15}^{(d+2)} \\
b_{21}^{(d+2)} & \times & b_{23}^{(d+2)} & b_{24}^{(d+1)} & b_{25}^{(d+1)} \\
0   & b_{32}^{(d+2)} & \times & b_{34}^{(d+2)} & b_{35}^{(d+2)} \\
b_{41}^{(d+2)} & b_{42}^{(d+1)} & b_{43}^{(d+2)} & \times & b_{45}^{(d)} \\
b_{51}^{(d+2)} & b_{52}^{(d+1)} & b_{53}^{(d+2)} & b_{54}^{(d)} & \times \\
\end{array}
\right]  \nonumber\\
&& \xrightarrow[]{G(1,4; \theta_{d+2})} \left[
\begin{array}{ccccc}
\times & b_{12}^{(d+3)} & b_{13}^{(d+3)} & 0 & b_{15}^{(d+3)} \\
b_{21}^{(d+3)} & \times & b_{23}^{(d+2)} & b_{24}^{(d+3)} & b_{25}^{(d+1)} \\
b_{31}^{(d+3)}  & b_{32}^{(d+2)} & \times & b_{34}^{(d+3)} & b_{35}^{(d+2)} \\
0 & b_{42}^{(d+3)} & b_{43}^{(d+3)} & \times & b_{45}^{(d+3)} \\
b_{51}^{(d+3)} & b_{52}^{(d+1)} & b_{53}^{(d+2)} & b_{54}^{(d+3)} & \times \\
\end{array}
\right] \nonumber\\
&&\xrightarrow[]{G(1,5; \theta_{d+3})} \left[
\begin{array}{ccccc}
\times & b_{12}^{(d+4)} & b_{13}^{(d+4)} & b_{14}^{(d+4)} & 0 \\
b_{21}^{(d+4)} & \times & b_{23}^{(d+2)} & b_{24}^{(d+3)} & b_{25}^{(d+4)} \\
b_{31}^{(d+4)}  & b_{32}^{(d+2)} & \times & b_{34}^{(d+3)} & b_{35}^{(d+4)} \\
b_{41}^{(d+4)} & b_{42}^{(d+3)} & b_{43}^{(d+3)} & \times & b_{45}^{(d+4)} \\
0 & b_{52}^{(d+4)} & b_{53}^{(d+4)} & b_{54}^{(d+4)} & \times \\
\end{array}
\right].
\end{eqnarray}
For the entries of the first row of $B^{(d+4)}$, we have the following inequalities
\BE\label{eq:sum1row}
\left\{
\begin{array}{l}
|b_{14}^{(d+4)}|\leq |b_{54}^{(d+3)}||\sin \theta_{d+3}|,\\[2mm]
|b_{13}^{(d+4)}|\leq |b_{43}^{(d+2)}||\sin\theta_{d+2}|+|b_{53}^{(d+2)}||\sin\theta_{d+3}|,\\[2mm]
|b_{12}^{(d+4)}|\leq |b_{32}^{(d+1)}||\sin\theta_{d+1}|+|b_{42}^{(d+1)} ||\sin\theta_{d+2}|+|b_{52}^{(d+1)}||\sin\theta_{d+3}|.
\end{array}
\right.
\EE
Thus,
\begin{eqnarray*}
&& |b_{12}^{(d+4)}|^2+ |b_{13}^{(d+4)}|^2+ |b_{14}^{(d+4)}|^2 \\
&\leq & (|b_{32}^{(d+1)}|^2+|b_{42}^{(d+1)}|^2+|b_{52}^{(d+1)}|^2)(\sin^2\theta_{d+1}+\sin^2\theta_{d+2}+\sin^2\theta_{d+3})\\
&&+(|b_{43}^{(d+2)}|^2+|b_{53}^{(d+2)}|^2)(\sin^2\theta_{d+2}+\sin^2\theta_{d+3})+|b_{54}^{(d+3)}|^2\sin^2\theta_{d+3}\\
&\leq &(|b_{32}^{(d+1)}|^2+|b_{42}^{(d+1)}|^2+|b_{52}^{(d+1)}|^2+|b_{43}^{(d+2)}|^2+|b_{53}^{(d+2)}|^2+|b_{54}^{(d+3)}|^2)\\
&&\times (\sin^2\theta_{d+1}+\sin^2\theta_{d+2}+\sin^2\theta_{d+3}).
\end{eqnarray*}
Since each rotation affects only two entries in each of the related  columns or rows while the sum of the squares of their absolute values is kept unchanged we have
\begin{eqnarray*}
\nonumber |b_{15}^{(d+3)}|^2+|b_{25}^{(d+1)}|^2+|b_{35}^{(d+2)}|^2+|b_{45}^{(d+3)}|^2 &= & |b_{15}^{(d)}|^2+|b_{25}^{(d)}|^2+|b_{35}^{(d)}|^2+|b_{45}^{(d)}|^2, \\
\nonumber |b_{14}^{(d+2)}|^2+|b_{24}^{(d+1)}|^2+|b_{34}^{(d+2)}|^2 &=& |b_{14}^{(d)}|^2+|b_{24}^{(d)}|^2+|b_{34}^{(d)}|^2,\\
|b_{13}^{(d+1)}|^2+|b_{23}^{(d+1)}|^2 &= & |b_{13}^{(d)}|^2+|b_{23}^{(d)}|^2.
\end{eqnarray*}
Thus
\begin{eqnarray}\label{eq:1row}
&&  |b_{12}^{(d+4)}|^2+ |b_{13}^{(d+4)}|^2+ |b_{14}^{(d+4)}|^2 \nonumber \\
&\leq &(|b_{13}^{(d)}|^2+|b_{23}^{(d)}|^2+|b_{14}^{(d)}|^2+|b_{24}^{(d)}|^2+|b_{34}^{(d)}|^2+|b_{15}^{(d)}|^2+|b_{25}^{(d)}|^2 +|b_{35}^{(d)}|^2+|b_{45}^{(d)}|^2) \nonumber \\
&&\times(\sin^2\theta_{d+1}+\sin^2\theta_{d+2}+\sin^2\theta_{d+3}) \nonumber \\
 &\leq&\frac{1}{2} \off(B^{(d)})^2(\sin^2\theta_{d+1}+\sin^2\theta_{d+2}+\sin^2\theta_{d+3}).
\end{eqnarray}
Moreover, the sum of the squares of the absolute values of these entries  in  the first row remains unchanged in the subsequent rotations.

Similarly, we have after successively  annihilating the entries in  the second row
\begin{eqnarray}\label{eq:2row}
 |b_{23}^{(d+7)}|^2+|b_{24}^{(d+7)}|^2 &\le& \frac{1}{2}\off(B^{(d+4)})^2(\sin^2 \theta_{d+5}+\sin^2 \theta_{d+6}) \nonumber\\
 &\le& \frac{1}{2} \off(B^{(d)})^2(\sin^2 \theta_{d+5}+\sin^2 \theta_{d+6}).
\end{eqnarray}
Furthermore, for the third row we have
\begin{eqnarray}\label{eq:3row}
|b_{34}^{(d+9)}|^2 \leq \frac{1}{2}\off(B^{(d+7)})^2\sin^2 \theta_{d+8}
\leq \frac{1}{2} \off(B^{(d)})^2\sin^2 \theta_{d+8}.
\end{eqnarray}
Finally, the fourth row above the diagonal is annihilated. From (\ref{eq:1row}), (\ref{eq:2row}) and (\ref{eq:3row}) we obtain
\begin{eqnarray}\label{eq:allrow}
 \off(B^{(d+10)})^2 \le \off(B^{(d)})^2 \sum_{t=0}^{9}\sin^2\theta_{d+t}
\le  \off(B^{(d)})^2  \frac{\off(B^{(d)})^2}{2\delta^2}
= \frac{\off(B^{(d)})^4}{2\delta^2},
\end{eqnarray}
where the second inequality follows from (\ref{eq:6}).

Analogously to the proof of  (\ref{eq:allrow}), using the equality $\off(S^{(d)})^2=4\off(B^{(d)})^2$, (\ref{eq:5}b) and (\ref{eq:6}) we have
\begin{eqnarray*}
\off(S^{(d+N)})^2 &=& 4\off(B^{(d+N)})^2 \\
&=&  4\sum_{1\le p\neq q\le n_1}|b_{pq}^{(d+N)}|^2+8\sum_{p<q, q>n_1}|b_{pq}^{(d+N)}|^2\\
&\leq& \off(\widehat{S}_{11}^{(d+N)})^2+4\off(B^{(d)})^2\Big(\widehat{\sum} \sin^2\theta_\ell\Big)\\
&\leq& \off(\widehat{S}_{11}^{(d+N)})^2+4\off(B^{(d)})^2\frac{\off(B^{(d)})^2}{2\delta^2}\\
&\leq& \frac{\off(S^{(d)})^4}{9\delta^2}+\frac{\off (S^{(d)})^4}{8\delta^2}.
\end{eqnarray*}
This shows that  Algorithm \ref{alg:2} converges quadratically when there is only one multiple singular value.

Next, we show the quadratic convergence of Algorithm \ref{alg:2} for the case of more than one multiple singular values. If there exist $l$ multiple singular values, then we have
\begin{eqnarray}\label{eq:leigs}
\off(S^{(d+N)})^2 \le \frac{9+8l}{72\delta^2} \off(S^{(d)})^4.
\end{eqnarray}

In the following we decrease the factor $\frac{9+8l}{72}$. Assume that $\sigma_{\varphi_w}$ is a multiple singular value of $A^{(\ell)}$ with multiplicity $n_w$ for $w=1, \ldots, l$ and  $s_{11}^{(\ell)}, \ldots, s_{n_1n_1}^{(\ell)}$, $s_{n_1+1,n_1+1}^{(\ell)}, \ldots, s_{n_1+n_2,n_1+n_2}^{(\ell)}$, $\ldots, $ $s_{n_1+\cdots + n_{l-1}+1,n_1+\cdots + n_{l-1}+1}^{(\ell)}, \ldots, s_{n_1+ \cdots + n_{l},n_1+ \cdots + n_{l}}^{(\ell)}$ converge to $\sigma_{\varphi_1}^2$, $\sigma_{\varphi_2}^2$, $\ldots$ , $\sigma_{\varphi_l}^2$ accordingly. Then, by appropriate row and column interchanges, we get a permutation matrix $P$ such that
\[
\widehat{S}^{(\ell)}=P^TS^{(\ell)}P=\left[
\begin{array}{ccccc}
\widehat{S}_{11}^{(\ell)} & \widehat{S}_{12}^{(\ell)} & \cdots & \widehat{S}_{1l}^{(\ell)} & \widehat{S}_{1,l+1}^{(\ell)} \\
\widehat{S}_{21}^{(\ell)} & \widehat{S}_{22}^{(\ell)} &\cdots &\widehat{S}_{2l}^{(\ell)} & \widehat{S}_{2,l+1}^{(\ell)} \\
 \vdots &  \vdots & \ddots &  \vdots &  \vdots \\
\widehat{S}_{l1}^{(\ell)} & \widehat{S}_{l2}^{(\ell)} & \cdots &\widehat{S}_{ll}^{(\ell)} & \widehat{S}_{l,l+1}^{(\ell)}\\
\widehat{S}_{l+1,1}^{(\ell)} & \widehat{S}_{l+1,2}^{(\ell)} & \cdots &\widehat{S}_{l+1,l}^{(\ell)} &\widehat{S}_{l+1,l+1}^{(\ell)} \\
\end{array}
\right],
\]
where $\widehat{S}_{ww}^{(\ell)}\in\mathbb{R}^{4n_w\times 4n_w}$ with diagonal entries converging to  $\sigma_{\varphi_w}^2$ for $w=1, \ldots, l$.

Define the quantities
\[
\Phi_w^{(\ell)}:=\sqrt{\sum_{\sum_{u=0}^{w-1}4n_u+1\le p\neq q\le \sum_{u=0}^w4n_u}\left(\hat{s}_{p, q}^{(\ell)}\right)^2},\quad w=1,\ldots, l,
\]
where $n_0=0$.

Analogous to the proof of  (\ref{eq:5}a) we have
\begin{eqnarray*}
(\Phi_w^{(\ell)})^2
 \le \frac{2\off(S^{(\ell)})^2}{9\delta^2}\sum_{u=1,u\neq w}^{l+1}\|\widehat{S}_{wu}^{(\ell)}\|_F^2.
 \end{eqnarray*}
We note that
\[
\sum_{w=1}^l\Big(\sum_{u=1,u\neq w}^{l+1}\|\widehat{S}_{wu}^{(\ell)}\|_F^2\Big) \le\sum_{1\leq p\neq q\le 4n}(\hat{s}_{pq}^{(\ell)})^2=  \off(\widehat{S}^{(\ell)})^2=  \off(S^{(\ell)})^2.
\]
Hence,
\begin{eqnarray}
\nonumber \sum_{w=1}^l (\Phi_w^{(\ell)})^2 &\leq& \frac{2\off(S^{(\ell)})^2}{9\delta^2}\sum_{w=1}^l\Big(\sum_{u=1,u\neq w}^{l+1}\|\widehat{S}_{wu}^{(\ell)}\|_F^2\Big)\leq \frac{2\off(S^{(\ell)})^4}{9\delta^2} \leq  \frac{2\off(S^{(d)})^4}{9\delta^2}.
\end{eqnarray}
Therefore, (\ref{eq:leigs}) is reduced to
\begin{eqnarray}
\off(S^{(d+N)})^2 &\leq& \sum_{t=1}^l(\Phi_t^{(d+N)})^2+\off(B^{(d)})^2\Big(\widehat{\sum} \sin^2\theta_\ell\Big) \nonumber\\
&\leq& \frac{2\off(S^{(d)})^4}{9\delta^2}+\frac{\off (S^{(d)})^4}{8\delta^2}\nonumber\\
&\leq& \frac{25}{72}\cdot\frac{\off(S^{(d)})^4}{\delta^2}.
\end{eqnarray}
That is,
\[
\off(\Gamma^T_{A^{(d+N)}}\Gamma_{A^{(d+N)}})\le \sqrt{\frac{25}{72}}\cdot \frac{\off(\Gamma^T_{A^{(d)}}\Gamma_{A^{(d)}})^2}{\delta},
\]
which shows that Algorithm \ref{alg:2} is quadratically convergent when there exist more than one multiple singular values.

\end{proof}

Finally, we point out that, for an $m\times n$ quaternion matrix $A=A_0+A_1i+A_2j+A_3k\in\mathbb{H}^{m\times n}$, Algorithm \ref{alg:2} generates a matrix $\Gamma_{A^{(\Theta)}}$ after $\Theta$ orthogonal JRS-symplectic Jacobi updates such that
$A^{(\Theta)}=AV^{(\Theta)}$ has sufficiently orthogonal columns (which is measured by $\off(\Gamma_{A^{(\Theta)}}^T\Gamma_{A^{(\Theta)}})\le\tol\cdot \|\Gamma_A^T\Gamma_A\|_F$ for a prescribed tolerance $\tol>0$), where $V^{(\Theta)}=G^{(0)}G^{(1)}\cdots G^{(\Theta)}\in\mathbb{H}^{n\times n}$ is a unitary matrix. Then the SVD of $A$ follows from column scaling of $A^{(\Theta)}=AV^{(\Theta)}$, i.e.,
\BE\label{a:svd}
A^{(\Theta)}=AV^{(\Theta)}=U^{(\Theta)}\Sigma^{(\Theta)},
\EE
where $\Sigma^{(\Theta)}=\diag (\sigma_{\varphi_1}^{(\Theta)}, \sigma_{\varphi_2}^{(\Theta)}, \ldots, \sigma_{\varphi_n}^{(\Theta)})$ with $\sigma_{\varphi_w}^{(\Theta)}\ge 0$ for $w=1,\ldots,n$ and $U^{(\Theta)}\in\mathbb{H}^{m\times n}$ is such that $(U^{(\Theta)})^*U^{(\Theta)}=I_n$.

\section{Numerical Experiments}\label{sec4}
In this section, we present some numerical experiments to illustrate the effectiveness of Algorithm \ref{alg:2} for computing the SVD of a rectangle quaternion matrix and compare it with the quaternion toolbox for {\tt MATLAB} \cite{sbqtfm} and the implicit Jacobi algorithm in \cite{bisa07}. We also apply the proposed algorithm to color image compression.  All the numerical tests were carried out in  {\tt MATLAB R2016a} running on a laptop of 2.00 GHz CPU and 4GB of RAM.

\begin{example}\label{ex:1}\upshape
In this example we compute the SVD of 19 $m\times n$ random quaternion matrices, where $m$ ranges from 10 to 100 with increment of 5 and $n=m/5$.
\end{example}

Figures \ref{fig:cputimes} and \ref{fig:errors} show, respectively, the CPU time and the residual $\|AV^{(\Theta)}-U^{(\Theta)}\Sigma^{(\Theta)}\|_F$ at the final iterate of the corresponding algorithms for different quaternion matrix sizes.

We can see from Figure \ref{fig:cputimes} that Algorithm  \ref{alg:2} is more efficient than the implicit Jacobi algorithm in \cite{bisa07} as the matrix size becomes larger. We also observe from Figure \ref{fig:errors} that  both two algorithms obtain almost the same calculation accuracy.
\begin{figure}[!htb]
	\begin{minipage}[t]{0.5\linewidth}
	\centering
	\includegraphics[width=8.7cm]{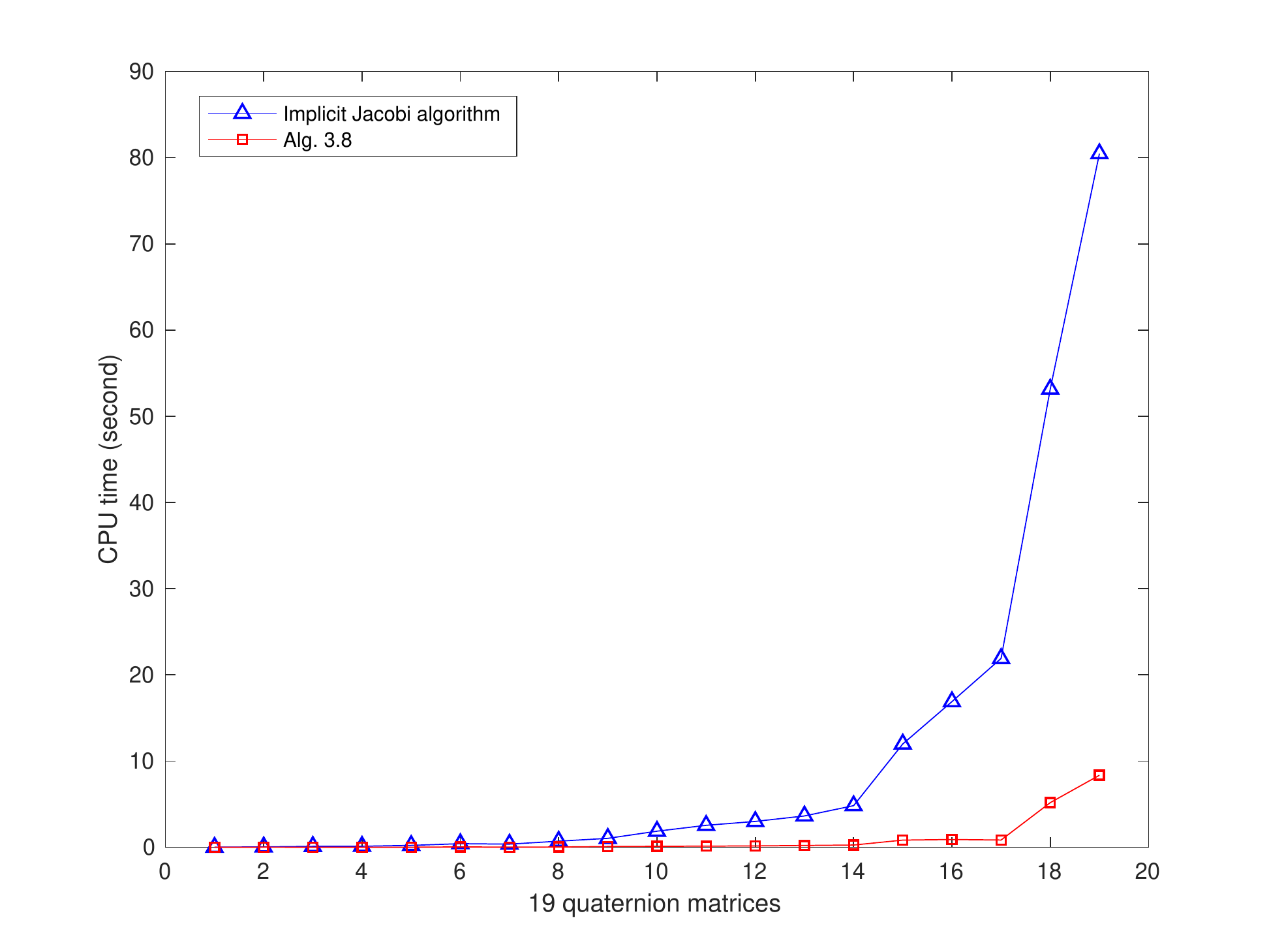}
	\caption{Numerical results for Ex. \ref{ex:1}.}\label{fig:cputimes}
	\end{minipage}
	\begin{minipage}[t]{0.5\linewidth}
	\centering
	\includegraphics[width=8.7cm]{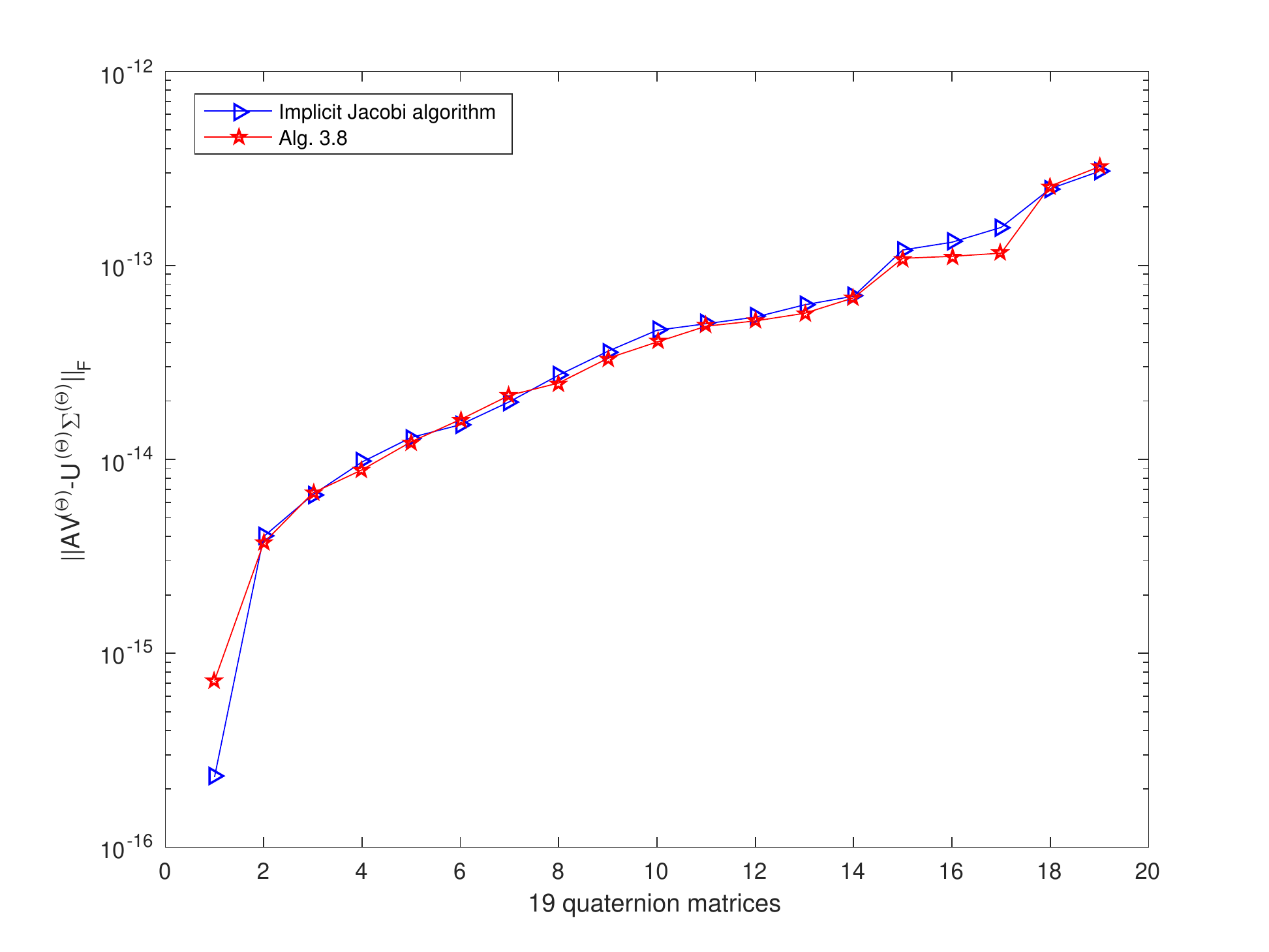}
	\caption{Numerical results for Ex. \ref{ex:1}.}\label{fig:errors}
\end{minipage}
\end{figure}

To further illustrate the effectiveness of Algorithm  \ref{alg:2}, Figure \ref{fig:singvals} depicts the computed singular values of a $100\times 20$ quaternion matrix by using Algorithm \ref{alg:2} and the {\tt svd} function in the quaternion toolbox for {\tt MATLAB} \cite{sbqtfm}. We see from Figure \ref{fig:singvals} that both algorithms obtain almost the same singular values.
\begin{figure}[!htb]
\centering
	\includegraphics[width=12cm]{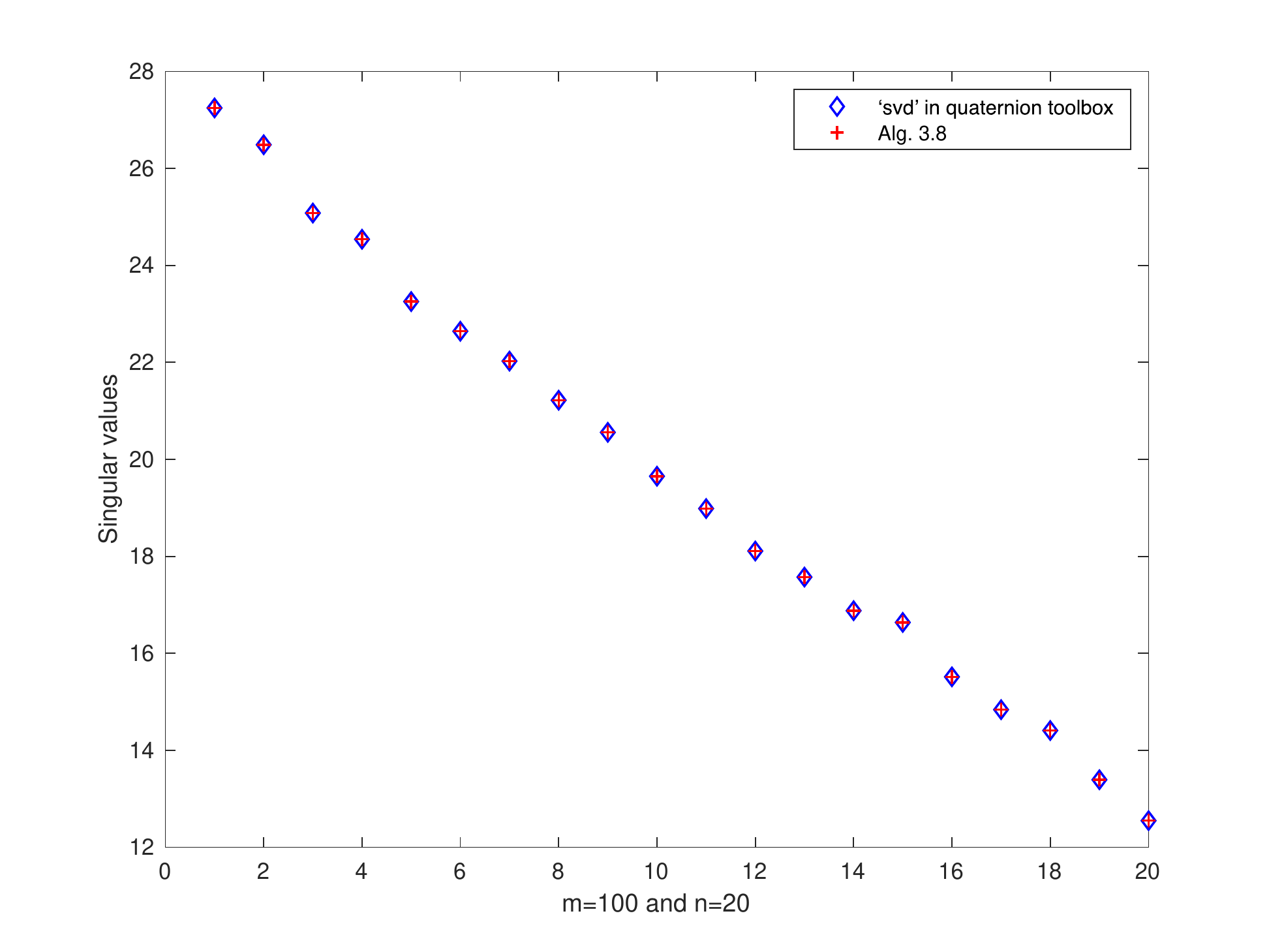}
	\caption{Computed singular values for one of the test quaternion matrices.}\label{fig:singvals}
\end{figure}

\begin{example}\label{ex:2}\upshape
In this example, we apply Algorithm \ref{alg:2} to image compression. We know a color image can be represented by a pure quaternion matrix $ A=(a_{ij})_{m\times n}=Ri+Gj+Bk $, where $R$, $G$, $B$ represent the red, green, blue parts of the color image. We use Algorithm  \ref{alg:2} to get the SVD of the color image without separating the color image into three channel images.
\end{example}

For demonstration purpose, in Example \ref{ex:2}, we take the color images  Snowberg, Rabbit and Eiffel Tower (Eiffel), whose sizes are  $50\times 50$, $50\times 50$ and $50\times 100$ accordingly. Figure \ref{fig:3singvals} shows the singular values of the original three color images. We can see that the singular values of these  images decay very fast. One may use Algorithm \ref{alg:2} to compute the SVD of a color image $A$ such that $AV^{(\Theta)}=U^{(\Theta)}\Sigma^{(\Theta)}$, where $U^{(\Theta)}$, $\Sigma^{(\Theta)}$, and $V^{(\Theta)}$ are given by (\ref{a:svd}). Then we can compress an image by  a lower-rank matrix approximation:
\begin{equation}
A_S= \sum_{w=1}^{S}\sigma_w^{(\Theta)}{\bu}_w^{(\Theta)}({\bv}_w^{(\Theta)})^*,
\end{equation}
where $\{\sigma_w^{(\Theta)}\}_{w=1}^S$ are the $S$ largest singular values of $A$, ${\bu}_w^{(\Theta)}$ and ${\bv}_w^{(\Theta)}$ are the left and right singular vectors of $A$ corresponding to $\sigma_w^{(\Theta)}$ for $w=1,\ldots,S$.

\begin{figure}[!htb]
\centering
	\includegraphics[width=12cm]{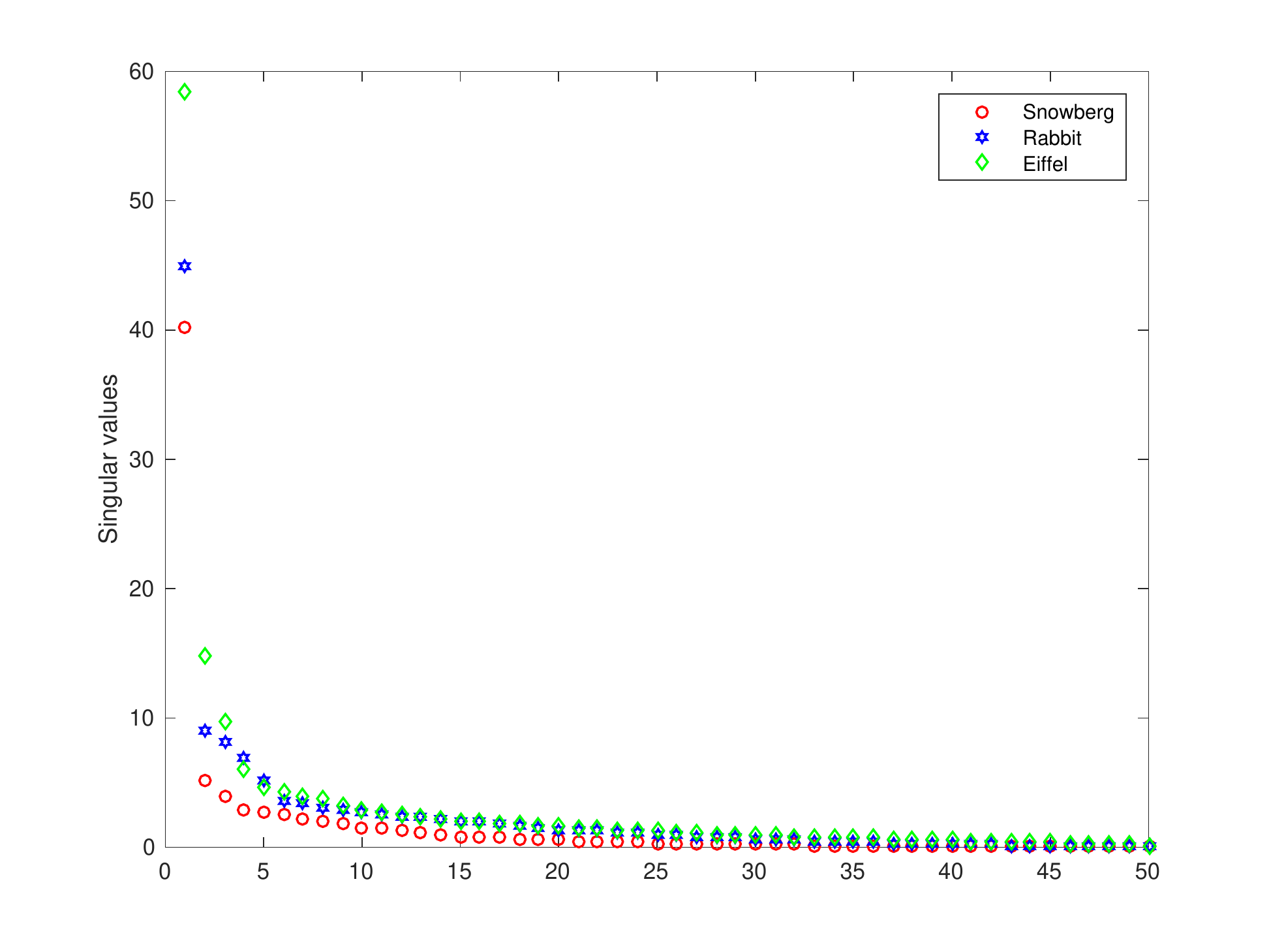}
	\caption{Singular values of three original color images.}\label{fig:3singvals}
\end{figure}

Figure \ref{fig:5figures} displays the original image and four estimated images with $S=10, 20, 30$ and $40$. We observe from  Figure \ref{fig:5figures}  that small $S$  already provides a good estimation of the original color image. Meanwhile, the storage requirements drop from $3mn$ to $S(4m+4n+1)$.  The peak signal-to-noise ratios (PSNRs) of the four estimated images are also listed in Table \ref{table:psnr}. The PSNR between the original image $f$ and a test image $g$, both of size $m\times n$, is defined by
\[
PSNR(f,g)=10 \log_{10}\left(\frac{255^2}{MSE(f,g)}\right),
\]
where MSE means the mean squared error defined by
\[
MSE(f,g)=\frac{1}{mn}\sum_{i=1}^{m}\sum_{j=1}^{n}(f_{ij}-g_{ij})^2.
\]

From Table \ref{table:psnr}, we  see that a lower-rank matrix approximation may provide a  good image compression, where the needed singular values and associated left and right singular vectors can be obtained by our algorithm.

\begin{figure}[!htb]
	\centering
	\subfigure[original]{
\begin{minipage}[b]{0.114\textwidth}
\includegraphics[width=\textwidth]{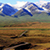}\\
\includegraphics[width=\textwidth]{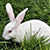}\\
\includegraphics[width=\textwidth]{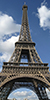}
\end{minipage}
}
	\subfigure[S=10]{
\begin{minipage}[b]{0.114\textwidth}
\includegraphics[width=\textwidth]{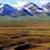}\\
\includegraphics[width=\textwidth]{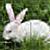}\\
	\includegraphics[width=\textwidth]{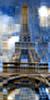}	
\end{minipage}
}
	\subfigure[S=20]{
\begin{minipage}[b]{0.114\textwidth}
\includegraphics[width=\textwidth]{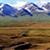}\\
\includegraphics[width=\textwidth]{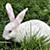}\\
	\includegraphics[width=\textwidth]{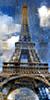}	
\end{minipage}
}
	\subfigure[S=30]{
\begin{minipage}[b]{0.114\textwidth}
\includegraphics[width=\textwidth]{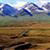}\\
\includegraphics[width=\textwidth]{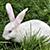}\\
	\includegraphics[width=\textwidth]{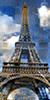}	
\end{minipage}
}
	\subfigure[S=40]{
\begin{minipage}[b]{0.114\textwidth}
\includegraphics[width=\textwidth]{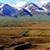}\\
\includegraphics[width=\textwidth]{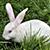}\\	
	\includegraphics[width=\textwidth]{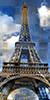}	
\end{minipage}
}
\caption{Original and estimated images.}\label{fig:5figures}
\end{figure}

\begin{table}[ht]\renewcommand{\arraystretch}{1.0} \addtolength{\tabcolsep}{2pt}
	\begin{center} {
			\begin{tabular}[c]{|c|c|c|c|c|}
				\hline
				S         & 10 & 20 & 30 & 40    \\  \hline
		Snowberg & 34.7060  &   35.9495  &  36.1030  &  36.1010    \\
				\cline{2-4}\hline		
		Rabbit	& 30.1050  &   33.8167  & 35.4383  & 35.6440   \\
				\cline{2-4}\hline
		Eiffel	 & 24.6439  &   25.4894  & 25.7858  & 25.8698 \\
				\cline{2-4}\hline
		\end{tabular} }
	\end{center}
	\caption{PSNR of the estimated color images.}\label{table:psnr}
\end{table}

\section{Conclusions}\label{sec5}
In this paper, we have proposed a real structure-preserving one-sided cyclic Jacobi algorithm for computing the QSVD. This  algorithm involves a sequence of column orthogonalizations in pairs via a sequence of orthogonal JRS-symplectic Jacobi rotations to the real counterpart of a quaternion matrix. The quadratic convergence is established under some assumptions. Finally, we report some numerical experiments to illustrate the efficiency of our algorithm. We also point out that the proposed algorithm can be used for  many practical applications such as the pseudoinverse of a quaternion matrix and color image processing (e.g., image compression, image enhancement, and image denoising).


\end{document}